\documentclass[12pt]{amsart}
\usepackage[english]{babel}
\usepackage[T1]{fontenc}
\usepackage{fullpage}
\usepackage{amsmath,amscd,amsthm,amsfonts,latexsym}

\theoremstyle{plain}
\newtheorem{theorem}                 {Theorem}      [section]
\newtheorem{proposition}  [theorem]  {Proposition}

\theoremstyle{definition}

\newtheorem{remark}       [theorem]  {Remark}

\DeclareMathOperator{\R}{R}
\DeclareMathOperator{\cst}{constant}
\DeclareMathOperator{\trace}{trace}

\def \r{\mbox{${\mathbb R}$}}

\def \co{\mbox{${\mathbb C}$}}
\def \h{\mbox{${\mathbb H}$}}

\def \s{\mbox{${\mathbb S}$}}
\def \S{\mbox{${\mathrm {SL}}(2,\r)_\tau$}}

\def \n{\mbox{${{\nabla}}^\tau$}}

\begin{document}

\title{On the biharmonic curves in the special linear group
${\mathrm {SL}}(2,\r)$}

%
%
\author{I. I. Onnis}
\address{Departamento de Matem\'{a}tica, C.P. 668\\ ICMC,
USP, 13560-970, S\~{a}o Carlos, SP\\ Brasil}
\email{onnis@icmc.usp.br}

\author{A. Passos Passamani}
\address{Departamento de Matem\'{a}tica, C.P. 668\\ ICMC,
USP, 13560-970, S\~{a}o Carlos, SP\\ Brasil}
\email{apoenapp@icmc.usp.br}

\date{March 2014}
\thanks{The second author was supported by Capes--Brazil}

\begin{abstract}
We characterize the biharmonic curves in the special linear group
$\mathrm{SL}(2,\r)$. In particular, we show that all proper biharmonic curves in
$\mathrm{SL}(2,\r)$ are helices and we give their explicit parametrizations as
curves in the pseudo-Euclidean space $\r^4_2$.
\end{abstract}

\maketitle

\section{Introduction}

Let $\phi:(M^m,g) \to (N^n,h)$ be a smooth map between two Riemannian manifolds.
The
tension field of $\phi$ is, by definition, $\tau(\phi)=\trace \nabla^N d(\phi)$.
According J.~Eells and J.H.~Sampson, see  \cite{EL},
$\phi$ is biharmonic if it is
a critical point of the bienergy functional

 \begin{equation}
 E_2(\phi)=\frac{1}{2}\int_{M}|\tau(\phi)|^2 v_g.
  \label{bienergy}
 \end{equation}

 The first variation formula for $E_2$ was compute by G.Y.~Jiang in \cite{Y1} and
 \cite{Y2} as
\begin{equation}
\tau_2(\phi):= -\Delta^{\phi} \tau (\phi)-\trace \R^N (d \phi,\tau(\phi)) d \phi=0,
\label{tau2}
\end{equation}
where $\Delta^{\phi}$ denotes the {\it rough Laplacian} acting on $C(\phi^{-1}TN)$, defined by
\begin{equation}
 \Delta^{\phi}= -\trace \,(\nabla^{\phi})^2=-\sum_{i=1}^m \{\nabla^{\phi}_{E_i}
 \nabla^{\phi}_{E_i}-\nabla^{\phi}_{\nabla^M_{E_i}{E_i}}\}.
\end{equation}
with respect to a local orthonormal frame field $\{E_i\}_{i=1}^m$ on $M$.

The field $\tau_2(\phi)$ is named {\it bitension field} of $\phi$. \\

As a geodesic curve ($\tau(\phi)=0$)
is a biharmonic one, we are interested in biharmonic curves that are not
geodesics i.e. {\it proper biharmonic curves}.
\\

The study of the proper biharmonic curves on a curved surface starts with
\cite{CMP} where there are described these curves
in a surface, proving that
biharmonic curves on a surface of non-positive Gaussian curvature are geodesics.

For $3$-dimensional Riemannian manifolds with constant sectional
curvature, the case of null and negative curvature are considered in \cite{D} and
\cite{CMO2} and it is showed
that the only biharmonic curves are the geodesic ones. Moreover, in \cite{CMO},
it is considered the case of positive curvature showing that
biharmonic curves have constant geodesic curvature and geodesic torsion.\\

 Besides the spaces forms,
 the most relevant $3$-dimensional homogeneous Riemannian spaces are those with
 $4$-dimensional isometry group: the Berger spheres,
 the Heisenberg group, the special linear group $\mathrm{SL}(2,\r)$, and the
 Riemannian product $\s^2 \times \r$ and $\h^2 \times \r$, where $\s^2$ and $\h^2$
 are the $2$-dimensional sphere and the hyperbolic plane, respectively.
A crucial feature of these spaces is that they admit a Rimannian submersion onto a
surface of constant Gaussian curvature, called the Hopf fibration.\\

 In \cite{B}
A.~Balmu{\c s} determined the parametric equations of all proper biharmonic curves
on the Berger sphere $\s^3_{\epsilon}$ as curves in $\r^4$ and gave a geometric
interpretation for those curves in the unit Euclidean sphere $\s^3$.
 In \cite{CPO} the authors proved that any proper biharmonic curve in the Heisenberg
  group is an helix and gave their explicit parametrizations.

Also,
in \cite{CMOP} are considered the proper biharmonic curves in the
Bianchi-Cartan-Vranceanu
spaces $\widetilde{\mathrm{SL}}(2,\r)$, $SU(2)$,
$\s^2 \times \r$ and $\h^2 \times \r$, proving that
these curves are helices and giving their parametric equations.\\


In this paper we study the proper biharmonic curves in the special linear group
$\mathrm{SL}(2,\r)$ endowed with a suitable $1$-parameter family $g_{\tau}$ of
metrics that we shall describe in Section~$2$.
Using the same technique given in \cite{B} for the case of the
Berger sphere, we conclude that the biharmonic curves of $\mathrm{SL}(2,\r)$ makes
a constant angle $\vartheta$ with the vector field tangent to the Hopf fibration
and we prove the Theorem~\ref{teoEDO}, which states that the differential
equation
$$\gamma^{IV}+(b^2-2a)\,\gamma''+a^2\,\gamma=0$$
must be satisfied by any proper biharmonic curve in $\mathrm{SL}(2,\r)$, as a curve
in the pseudo-Euclidean $\r^4_2$, where
$a$ and $b$ are real constants depending on $\vartheta$ and $\tau$.
We separate the study in three
cases depending on the sign of the constant $(b^2-4a)$ obtaining, in each case,
the expressions of these curves as curves in $\r^4_2$.

\section{Preliminaries}\label{preli}

Let $\r_2^4$ denote the $4$-dimensional pseudo-Euclidean space endowed with semi-definite inner product of signature $(2,2)$ given by
$$
\langle v,w\rangle= v_1\,w_1+v_2\,w_2-v_3\,w_3-v_4\,w_4\,,\quad v,w\in\r^4.
$$
We identify the special linear group with
$$
\mathrm{SL}(2,\r)=\{(z,w)\in\co^2 \colon |z|^2-|w|^2=1\}=\{v\in\r_2^4 \colon
\langle v,v\rangle=1\}\subset\r_2^4
$$
and we shall use the Lorentz model of the hyperbolic plane with constant Gauss
curvature $-\tau$, $\tau>0$, that is
$$
\h^2(-\tau)=\{(x,y,z)\in\r^3_1\colon x^2+y^2-z^2=-1/\tau\},
$$
where $\r^3_1$ is the Minkowski $3$-space. Then the Hopf map
$\psi:\mathrm{SL}(2,\r)\to\h^2(-\tau)$
given by
$$
\psi(z,w)=\frac{1}{\sqrt{\tau}}\,(2z\bar{w},|z|^2+|w|^2)
$$
is a submersion, with circular fibers, and if we put
$$
X_1(z,w)=(iz,iw),\quad X_2(z,w)=(i\bar{w},i\bar{z}),\quad X_3(z,w)=(\bar{w},\bar{z}),
$$
we have that $X_1$ is a vertical vector field while $X_2$, $X_3$ are horizontal. The vector $X_1$ is called the {\em Hopf vector field}.

We shall endow $\mathrm{SL}(2,\r)$ with the $1$-parameter family of metrics $g_\tau$, $\tau>0$, given by
$$
g_\tau(X_i,X_j)=\delta_{ij},\quad g_\tau(X_1,X_1)=\tau^2,\quad g_\tau(X_1,X_j)=0,\quad i,j\in\{2,3\},
$$
which renders the Hopf map $\psi:(\mathrm{SL}(2,\r),g_{\tau})\to\h^2(-\tau)$ a Riemannian submersion.
With respect to the inner product in $\r_2^4$ the metric $g_{\tau}$ is given by
\begin{equation}\label{eq-def-gtau}
g_{\tau}(X,Y)=-\langle X,Y\rangle+(1+\tau^2)\langle X, X_1\rangle \langle Y, X_1\rangle\,.
\end{equation}
From now on, we denote $(\mathrm{SL}(2,\r),g_\tau)$ with $\S$. Obviously
\begin{equation}\label{eq-basis}
    E_1=-\tau^{-1}\,X_1,\quad
     E_2=X_2,\quad
     E_3=X_3,
\end{equation}
is an orthonormal  basis on $\S$. The Levi-Civita connection $\n$ of $\S$ is
given by:
\begin{equation}
\begin{aligned}
&\n_{E_{1}}E_{1}=0,\quad \n_{E_{2}}E_{2}=0,\quad \n_{E_{3}}E_{3}=0,\\
&\n_{E_{1}}E_{2}=-\tau^{-1}(2+\tau^2)E_{3}, \quad \n_{E_{1}}E_{3}=\tau^{-1}(2+\tau^2)E_{2},\\
 &\n_{E_{2}}E_{1}=-\tau E_{3},\quad \n_{E_{3}}E_{1}=\tau E_{2},\quad
 \n_{E_{3}}E_{2}=-\tau E_{1}=-\n_{E_{2}}E_{3}.
\end{aligned}
\label{nabla}
\end{equation}
Using the conventions
$$\R (X,Y)Z= \n_X \n_Y Z-\n_Y \n_X Z-\n_{[X,Y]} Z$$
and
$$\R (X,Y,W,Z)= g_{\tau}(\R(X,Y)Z,W),$$
the nonzero components of the Riemannian curvature are
\begin{equation}
\R_{1212}=\tau^2, \qquad \R_{1313}=\tau^2,  \qquad \R_{2323}=-(4+3\tau^2),
\label{curvature}
\end{equation}
where $\R_{ijkl}=\R(E_i,E_j,E_k,E_l)$. \\

Finally, we recall that the isometry group of $\S$ is the $4$-dimensional indefinite unitary group $\mathrm{U}_1(2)$
that can be identified with:
$$
\mathrm{U}_1(2)=\{A\in \mathrm{O}_2(4)\colon AJ_1=\pm J_1A\}\,,
$$
where $J_1$ is the complex structure of $\r^4$ defined by
$$
J_1 = \left(\begin{matrix}J & 0 \\ 0 & J\end{matrix}\right)\,,\quad J = \left(\begin{matrix}0 & -1 \\ 1 & 0\end{matrix}\right)\,,
$$
while
$$\mathrm{O}_2(4)=\{A\in \mathrm{GL}(4,\r)\colon A^t=\epsilon\,A^{-1}\,\epsilon\},\qquad \epsilon=\begin{pmatrix}I&0\\0&-I\end{pmatrix},\quad I=\begin{pmatrix}1&0\\0&1\end{pmatrix},$$
is the indefinite orthogonal group.

\section{Biharmonic curves in $\S$}
Let $\gamma: I \to \S$ be a differentiable curve parametrized by arc
length and let $\{T, N, B\}$ be the orthonormal frame field tangent to
$\S$ along $\gamma(s)$ defined as follows:  we denote by $T$ the unit
vector field
$\gamma'(s)$ tangent to $\gamma(s)$, by $N$ the unit vector
field in the direction of $\n_{T} T$ normal to $\gamma$, and we choose
$B$ so that
$\{T, N, B\}$ is a positive oriented orthonormal basis.
Then we have the following
Frenet equations
\begin{equation}
\left\{
\begin{aligned}
 &\n_{T} T= k_1 N, \\
 &\n_{T} N=-k_1 T +k_2 B,\\
 &\n_{T} B=-k_2 N,
 \end{aligned}
 \label{frenet}
 \right.
\end{equation}
where $k_1=|\n_{T} T|$ is the geodesic curvature of $\gamma$ and $k_2$
its geodesic torsion.

 \begin{theorem}
 Let $\gamma:I \to \S$ be a curve parametrized by arc length. Then $\gamma$ is
 proper biharmonic if and only if
\begin{equation}
  \left\{\begin{aligned}
& k_1 = \cst \neq 0 ,\\
& k_1^2+k_2^2 =\tau^2-4\,(1+\tau^2)\,B_1^2,\\
& k_2' =-4(1+\tau^2)\,N_1\, B_1.
\label{7.1}
\end{aligned}
\right.
 \end{equation}

 \end{theorem}
\begin{proof}

Consider a curve $\gamma: I \to \S$ parametrized by arc length. In this case
the equation \eqref{tau2} becomes
\begin{equation}
  (\n_T)^3 T-\R(T,\n_T T)T=0.
  \label{eqbi}
 \end{equation}
Using the Frenet equations into \eqref{eqbi}, we obtain the
 conditions
\begin{equation}
  \left\{\begin{aligned}
& k_1 = \cst \neq 0 ,\\
& k_1^2+k_2^2 =\R(T,N,T,N),\\
& k_2' =-\R(T,N,T,B).
\label{sistemaBi1}
\end{aligned}
\right.
 \end{equation}
Writing
 \begin{equation}
 \label{TNB}
T=\sum_{i=1}^3 T_i\, E_i,\qquad N=\sum_{i=1}^3 N_i\, E_i,\qquad
B=\sum_{i=1}^3 B_i\, E_i,
 \end{equation}
and using \eqref{curvature}, we have that
\begin{equation*}
\begin{aligned}
 &\R(T,N,T,N)=\tau^2-4(1+\tau^2)B_1^2,\\
 &\R(T,N,T,B)=4(1+\tau^2)N_1\,B_1.
\end{aligned}
\end{equation*}
\end{proof}

\begin{proposition}\label{curvaturasCST}
 If $\gamma:I \to \S$ is a proper biharmonic curve parametrized by arc length, then
 its geodesic curvature and torsion are constants.
\end{proposition}
\begin{proof}
 From the Frenet equations it results that
 $$g_{\tau}(\n_{T} B,E_1)=-g_{\tau}(k_2 N,E_1)=-k_2\, N_1.$$
 On the other hand, using \eqref{nabla}, we get
\begin{equation*}
\begin{aligned}
 g_{\tau}(\n_{T} B, E_1)&=  g_{\tau}(B'_1\, E_1+T_2\, B_3
 \n_{E_2} E_3+T_3 B_2
 \n_{E_3} E_2, E_1)\\
 &=B'_1+\tau(T_2\, B_3-T_3\, B_2)\\
 &=B'_1-\tau N_1.
 \end{aligned}
\end{equation*}
Combining these two equations, we have
\begin{equation}
 B'_1=(\tau-k_2)\, N_1.
 \label{8.1}
\end{equation}
Now, using \eqref{7.1} we obtain
\begin{equation}
 k_2\, k'_2=-4(1+\tau^2)\,B_1\, B'_1.
 \label{8.2}
\end{equation}
From \eqref{8.1} and \eqref{8.2} it results that $(\tau-2k_2) B_1\, N_1=0$.
Therefore, we have two
possibilities: $B_1\, N_1=0$ that, together with \eqref{7.1}, implies $k'_2=0$;
or $k_2=\frac{\tau}{2}$. So $k_2$ is constant.
\end{proof}

\begin{proposition}
 If $\gamma: I \to \S$ is a proper biharmonic curve parametrized  by arc length,
 then it makes a constant angle with the Hopf vector field $E_1$ and its tangent
 vector field can be writen as
 \begin{equation}
\gamma'(s)=T=\cos\vartheta\, E_1+ \sin\vartheta \sin\beta(s)\, E_2+
\sin\vartheta \cos\beta(s)\, E_3,
\label{10.2}
\end{equation}
\end{proposition}
where $\vartheta\in(0,\pi/2]$ and $\beta: I \to \r$ is a smooth function.
\begin{proof}
 First we note that $B_1\neq 0$.
Indeed if $B_1=0$ and $N_1=0$, then the curve is the integral curve of
 the vector
 field $E_1$ and it is a geodesic. Moreover, if $B_1=0$ and $N_1 \neq 0$, from
 \eqref{8.1} we get $k_2=\tau$ that, together with the second equation of
 \eqref{7.1}, gives $k_1=0$.

 Since $B_1\neq 0$, the third equation of \eqref{7.1} and the
 Proposition~\ref{curvaturasCST} implies $N_1=0$.
Now, using
 the equations \eqref{nabla} and \eqref{frenet} we
 obtain
\begin{equation*}
 k_1N_1= g_{\tau}(\n_{T}T,E_1)=T'_1.
\end{equation*}
We conclude that $T_1=\cst$ and we obtain the expression \eqref{10.2}.
\end{proof}

Using the previous result we have the following
\begin{theorem} \label{teoEDO}
 Let $\gamma: I \to \S\subset \r^4_2$ be a curve parametrized by arc length.
 Then $\gamma$ is
 proper biharmonic if and only if, as a curve in $\r^4_2$, satisfies
 \begin{equation}
  \label{11.1}
  \gamma^{IV}+(b^2-2a)\,\gamma''+a^2\,\gamma=0,
 \end{equation}
where $a$ and $b$ are the constants given by:
\begin{equation} \label{ab}
 \left\{
 \begin{aligned}
 &a=\frac{1}{2}(-\tau^{-2}+1-(1+\tau^{-2})\cos 2\vartheta)-
  \tau^{-1}\cos \vartheta\, \beta',\\
  &b=\beta'= -\tau^{-1}(2+\tau^2)\cos\vartheta\pm \sqrt{(4+5\tau^2)\cos^2\vartheta
  -4(1+\tau^2)},
 \end{aligned}
\right.
\end{equation}
with $$ \frac{4(1+\tau^2)}{(4+5\tau^2)}\leq \cos^2\vartheta<1. $$
\end{theorem}

\begin{proof}
Writing
 $$
    \gamma(s)=(x_1(s),x_2(s),x_3(s),x_4(s)),
$$
from \eqref{10.2} we have that the coordinates functions of $\gamma$ in $\r_2^4$ satisfies
\begin{equation}\label{coorT}\left\{\begin{aligned}
x_1'&=\tau^{-1}\cos\vartheta\,x_2+\sin\vartheta\cos\beta\,x_3+
\sin\vartheta\sin\beta\,x_4,\\
x_2'&=-\tau^{-1}\cos\vartheta\,x_1+\sin\vartheta\sin\beta\,x_3-
\sin\vartheta\cos\beta\,x_4,\\
x_3'&=\sin\vartheta\cos\beta\,x_1+\sin\vartheta\sin\beta\,x_2+
\tau^{-1}\cos\vartheta\,x_4,\\
x_4'&=\sin\vartheta\sin\beta\,x_1-\sin\vartheta\cos\beta\,x_2-
\tau^{-1}\cos\vartheta\,x_3.\\
\end{aligned}
\right.
\end{equation}
Deriving \eqref{coorT}, it results that
\begin{equation}\label{coorT'}\left\{\begin{aligned}
x_1''&= a\, x_1- b\, x_2',\\
x_2''&= a\, x_2+ b\, x_1',\\
x_3''&= a\, x_3- b\, x_4',\\
x_4''&= a\, x_4+ b\, x_3',\\
\end{aligned}
\right.
\end{equation}
where
$$
 \left\{
 \begin{aligned}
 &a=\frac{1}{2}(-\tau^{-2}+1-(1+\tau^{-2})\cos 2\vartheta)-
  \tau^{-1}\cos \vartheta\, \beta',\\
  &b=\beta'.\\
 \end{aligned}
\right.
$$

 Now, we shall prove that $b$ is constant and we determine its expression.
Computing $\n_TT$, using \eqref{10.2} and \eqref{nabla}, the geodesic curvature
and the normal vector field are given by
\begin{equation}\label{k1N}
 k_1=\pm\sin\vartheta (\beta'+  \, 2\tau^{-1}\,(1+\tau^2)\cos\vartheta),\qquad
 N=\pm(\cos\beta\, E_2-\sin\beta\, E_3).
 \end{equation}
Then
\begin{equation}\label{k2B}
\begin{aligned}
  &B=T\wedge N=\pm(-\sin\vartheta\, E_1 +\cos\vartheta\sin\beta \,
  E_2+\cos\vartheta\cos\beta \, E_3),\\
  &k_2=g_{\tau}(\n_{T} N, B)=(\tau-\cos\vartheta
 (\beta'+ 2\tau^{-1}\,(1+\tau^2)\,\cos\vartheta )).
 \end{aligned}
  \end{equation}
Substituting the expressions of $k_1$, $k_2$ and $B_1$ in the second equation of
 \eqref{7.1},
 it results that
 $$\beta'= -\tau^{-1}(2+\tau^2)\cos\vartheta\pm \sqrt{(4+5\tau^2)\cos^2\vartheta
  -4(1+\tau^2)}.$$
Now deriving twice \eqref{coorT'}, and use \eqref{coorT}, we obtain the
 equation \eqref{11.1}.
Also, as the curve $\gamma$ is not harmonic, from \eqref{k1N}, $\cos\vartheta \neq 1$.

 \end{proof}
\begin{remark}
Using \eqref{coorT} and \eqref{coorT'},  we find that:
\begin{equation}\label{relacoes}
\begin{array}{lll}
\langle \gamma,\gamma\rangle=1\,,&\langle\gamma',\gamma'\rangle=\tilde{B},&
\langle \gamma,\gamma'\rangle=0,\\
\langle\gamma',\gamma''\rangle=0\,,& \langle \gamma'',\gamma''\rangle=D\,
,& \langle \gamma,\gamma''\rangle=-\tilde{B},\\
\langle \gamma',\gamma'''\rangle=-D \,,&
\langle \gamma'',\gamma'''\rangle=0\,,&  \langle \gamma,\gamma'''\rangle=0,\\
\langle \gamma''',\gamma'''\rangle=E,\,&\qquad &
\end{array}
\end{equation}
where $$\tilde{B}=(1+\tau^{-2})\cos^2\vartheta-1,
\qquad D=a^2+b^2\tilde{B}+2 \,a \,b\, \tau^{-1}\cos\vartheta,$$ $$
E=a\big(a-2b^2\big)\tilde{B}+b^2 D-2a^2b\, \tau^{-1}\cos\vartheta.$$
In addition, as
$$
J_1\gamma={X_1}_{|\gamma}=-\tau\,{E_1}_{|\gamma},
$$
using \eqref{10.2} and \eqref{coorT'},  we obtain the following identities
\begin{equation}\label{relacoesJ}
\begin{aligned}
&\langle J_1\gamma,\gamma'\rangle=-\tau^{-1}\cos\vartheta,\\
&\langle J_1 \gamma,\gamma''\rangle=0\,,\\
&\langle J_1\gamma'',\gamma'\rangle=-a \,\tau^{-1} \,\cos\vartheta-b\,
\tilde{B}:=I,\\
&\langle J_1\gamma',\gamma'''\rangle=0\,,\\
&\langle J_1\gamma',\gamma''\rangle+\langle J_1\gamma,\gamma'''\rangle=0\,,\\
&\langle J_1\gamma'',\gamma'''\rangle+\langle J_1\gamma',\gamma^{IV}\rangle=0\,.
\end{aligned}
\end{equation}
 \end{remark}

To determine the expression of the position vector of $\gamma$ in $\r_2^4$, we
 integrate \eqref{11.1}, dividing the study in three
cases, according to the three possibilities:

\begin{itemize}
\item [(i)] $b^2=4a$;


\item[(ii)] $b^2>4a$;


%
 \item[(iii)] $b^2<4a$.

%
%
%
%

\end{itemize}

\section{The case $b^2=4a$}
\begin{theorem}
 Let $\gamma: I\to \S \subset \r_2^4$ be a proper biharmonic curve parametrized by
 arc length such that $b^2=4a.$
 Then
 \begin{equation}
 \label{bcaso1}
b= -\tau^{-1}(2+\tau^2)\cos\vartheta+\sqrt{(4+5\tau^2)\cos^2\vartheta
  -4(1+\tau^2)},
\end{equation}

with

$$\cos^2\vartheta=\frac{(2+\tau^2)^2}{4+5\tau^2+\tau^4}.$$
Also,
  \begin{equation}
   \label{gammaCaso1.2}
   \begin{aligned}
   \gamma(s)=& A \Big(\cos(\sqrt{a}\,s)+g_{14}\,s\, \sin(\sqrt{a}\,s),
   -\sin(\sqrt{a}\,s)+g_{14}\, s\,\cos(\sqrt{a}\,s),\\
   &- g_{14}\, s\,\cos(\sqrt{a}\,s),
   \,   g_{14}\, s\,\sin(\sqrt{a}\,s)\Big),
 \end{aligned}
 \end{equation}
  where $g_{14}$ is the constant, given by
  $$g_{14}=\frac{\tau}{\sqrt{4+5\tau^2+\tau^4}}$$
  and
  $A \in \mathrm{O}_2(4)$ is a $4\times 4$ indefinite orthogonal matrix which
  commutes with $J_1$.

\end{theorem}
\begin{proof}
 As $b^2=4a$, the differential equation \eqref{11.1} turns
 \begin{equation}
  \label{50.1}
  \gamma^{IV}(s)+ 2 a \, \gamma''(s)+a^2 \gamma(s)=0.
 \end{equation}
Integrating \eqref{50.1} we have
\begin{equation}
\label{50.2}
  \gamma(s)= \cos(\sqrt{a}\,s)\, g_1+ \sin(\sqrt{a}\,s)\, g_2
  + s\,\cos(\sqrt{a}\,s)\, g_3  + s\,\sin(\sqrt{a}\,s)\, g_4,
 \end{equation}
where $g_1$, $g_2$, $g_3$ and $g_4$ are constant vectors of $\r^4_2$.\\

A direct
calculation shows that $b^2=4a$ occurs in two cases: for $\vartheta=0$ and
 for $$\cos^2\vartheta=\frac{(2+\tau^2)^2}{4+5\tau^2+\tau^4},$$
and in both cases $b$ must have the expression given in \eqref{bcaso1}. Since
the first case produces harmonic curves, we study only the second one.\\

Using the relations \eqref{relacoes}, we get
\begin{equation}
  \label{relacoes02}
  \begin{aligned}
  &\langle g_1, g_1 \rangle=\langle g_2, g_2 \rangle=1,\\
  &\langle g_3, g_3 \rangle=\langle g_4, g_4 \rangle=0,\\
  &\langle g_1, g_4 \rangle=-\langle g_2, g_3 \rangle=
  \frac{\tau}{\sqrt{4+5\tau^2+\tau^4}},\\
 &\langle g_1, g_2 \rangle=\langle g_1, g_3 \rangle=
 \langle g_2, g_4 \rangle=
\langle g_3, g_4 \rangle=0,
  \end{aligned}
 \end{equation}
whereas \eqref{relacoesJ} yields
\begin{equation}
  \label{relacoesJ02}
  \begin{aligned}
  &\langle J_1 g_1, g_2 \rangle=-1,\\
  &\langle J_1 g_2, g_4 \rangle=\langle J_1 g_1, g_3 \rangle=
  \frac{\tau}{\sqrt{4+5\tau^2+\tau^4}},\\
  &\langle J_1 g_1, g_4 \rangle=
  \langle J_1 g_2, g_3 \rangle=\langle J_1 g_3, g_4 \rangle =0.
  \end{aligned}
 \end{equation}
 Now, putting
$$
  \left\{
  \begin{aligned}
   &e_1=g_1,\\
   &e_2=g_2,\\
   &e_3=\frac{g_3}{\langle g_2,g_3\rangle}-g_2,\\
   &e_4=\frac{g_4}{\langle g_1,g_4\rangle}-g_1,
  \end{aligned}
  \right.
$$
we have that $\{e_i\}$ is an orthonormal basis of $\r^4_2$ that satisfies:
$$
\begin{aligned}
&\langle J_1 e_1,e_2\rangle=\langle J_1e_3,e_4\rangle=-1,\\
&\langle J_1 e_1,e_3\rangle=\langle J_1 e_1,e_4\rangle=\langle J_1 e_2,e_3\rangle=
 \langle J_1 e_2,e_4\rangle=0.
\end{aligned}
$$
We conclude that $e_2=-J_1 e_1$ and $e_4=J_1 e_3$.
So if we consider the orthonormal basis $\{\tilde{E}_i\}_{i=1}^4$ of $\r^4_2$
given by
$$
\tilde{E}_1=(1,0,0,0)\,,\quad \tilde{E}_2=(0,-1,0,0)\,,
\quad \tilde{E}_3=(0,0,1,0)\,,\quad \tilde{E}_4=(0,0,0,1)\,,
$$
there must exists a matrix $A\in \mathrm{O}_2(4)$, with $J_1\,A=A\, J_1$ such that
$e_i=A\, \tilde{E}_i$, $i~\in~\{1,2,3,4\}$. Finally, putting
$\langle g_1, g_4\rangle=g_{14}$,
we can rewrite \eqref{50.2} as
\eqref{gammaCaso1.2}.
\end{proof}

\section{The case $b^2>4a$}

%
%

\begin{theorem}
  Let $\gamma:I \to \S\subset \r^4_2$ be a proper biharmonic curve parametrized
  by arc length,
 such that $b^2>4a$. Then there are two possibilities:
 \begin{itemize}
  \item [(i)] $$b= -\tau^{-1}(2+\tau^2)\cos\vartheta+\sqrt{(4+5\tau^2)
  \cos^2\vartheta
  -4(1+\tau^2)}$$
  and $$  \frac{4(1+\tau^2)}{(4+5\tau^2)}\leq \cos^2\vartheta
  <\frac{(2+\tau^2)^2}{4+5\tau^2+\tau^4};$$

  \item[(ii)]
 $$b= -\tau^{-1}(2+\tau^2)\cos\vartheta-\sqrt{(4+5\tau^2)\cos^2\vartheta
  -4(1+\tau^2)}$$
  and
  $$ \frac{4(1+\tau^2)}{(4+5\tau^2)} \leq \cos^2\vartheta.$$
 \end{itemize}
In both cases, the expression of $\gamma$ as a curve in $\r^4_2$ is
\begin{equation}
\label{gammaCaso2.2}
 \gamma(s)=A\big(\sqrt{C_{33}}\cos(\alpha_2\,s)\,,\,
 \sqrt{C_{33}}\sin(\alpha_2\,s) \,,\,
\sqrt{-C_{11}}\cos(\alpha_1\,s)\,,\,\sqrt{-C_{11}}\sin(\alpha_1\,s)\big)\,,
\end{equation}
where
$$\alpha_{1,2}=\sqrt{\frac{(b^2-2a)\pm \sqrt{b^2(b^2-4a)}}{2}} $$
and
$$C_{11}=\frac{\tilde{B}-\alpha^2_2}{\alpha^2_1-\alpha^2_2}, \qquad
 C_{33}= \frac{-\tilde{B}+\alpha^2_1}{\alpha^2_1-\alpha^2_2},$$
are real constants and $A\in \mathrm{O}_2(4)$ is a $4\times 4$ indefinite
orthogonal matrix anticommuting with~$J_1$.
\end{theorem}

\begin{proof}
First, observe that the condition $b^2>4a$ gives the two possibilities (i) and (ii).
Also a direct integration of \eqref{11.1}, gives the solution
$$
\gamma(s)=\cos(\alpha_1\,s)\, C_1+\sin(\alpha_1\,s)\, C_2+
 \cos(\alpha_2\,s)\, C_3+\sin(\alpha_2\,s)\, C_4,$$
where
$$
\alpha_{1,2}=\sqrt{\frac{(b^2-2a)\pm \sqrt{b^2(b^2-4a)}}{2}}$$
are real constants, while the $C_i$, $i\in\{1,2,3,4\}$,  are constants vectors in
$\r^4_2$.
\\

Putting $C_{ij}=\langle C_i,C_j\rangle$, and evaluating  the relations
\eqref{relacoes}  in $s=0$, we obtain:
\begin{equation}\label{um}
    C_{11}+C_{33}+2C_{13}=1,
\end{equation}
\begin{equation}\label{dois}
    \alpha_1^2\,C_{22}+\alpha_2^2\,C_{44}+2\alpha_1\alpha_2\,C_{24}=\tilde{B},
\end{equation}
\begin{equation}\label{tres}
    \alpha_1\,C_{12}+\alpha_2\,C_{14}+\alpha_1\,C_{23}+\alpha_2\,C_{34}=0,
\end{equation}
\begin{equation}\label{quatro}
    \alpha_1^3\,C_{12}+\alpha_1\alpha_2^2\,C_{23}+\alpha_1^2\alpha_2\,C_{14}
    +\alpha_2^3C_{34}=0,
\end{equation}
\begin{equation}\label{cinco}
    \alpha_1^4\,C_{11}+\alpha_2^4\,C_{33}+2\alpha_1^2\alpha_2^2\,C_{13}=D,
\end{equation}
\begin{equation}\label{seis}
    \alpha_1^2\,C_{11}+\alpha_2^2\,C_{33}+(\alpha_1^2+\alpha_2^2)\,C_{13}=
    \tilde{B},
\end{equation}
\begin{equation}\label{sete}
    \alpha_1^4\,C_{22}+(\alpha_1^3\alpha_2\,+\alpha_1\alpha_2^3)\,C_{24}+
    \alpha_2^4\,C_{44}=D,
\end{equation}
\begin{equation}\label{oito}
    \alpha_1^5\,C_{12}+\alpha_1^3\alpha_2^2\,C_{23}+\alpha_1^2\alpha_2^3\,C_{14}
    +\alpha_2^5\,C_{34}=0,
\end{equation}
\begin{equation}\label{nove}
    \alpha_1^3\,C_{12}+\alpha_1^3\,C_{23}+\alpha_2^3\,C_{14}+\alpha_2^3\,C_{34}=0,
\end{equation}
\begin{equation}\label{dez}
    \alpha_1^6\,C_{22}+\alpha_2^6\,C_{44}+2\alpha_1^3\alpha_2^3\,C_{24}=E.
\end{equation}

From \eqref{tres}, \eqref{quatro}, \eqref{oito}, \eqref{nove}, it follows that
$$C_{12}=C_{14}=C_{23}=C_{34}=0.$$
Also, from  \eqref{um}, \eqref{cinco} and \eqref{seis}, we obtain
$$C_{11}=\frac{\tilde{B}-\alpha^2_2}{\alpha^2_1-\alpha^2_2},
\qquad C_{13}=0,\qquad C_{33}=\frac{-\tilde{B}+\alpha^2_1}{\alpha^2_1-\alpha^2_2}.$$
Finally, using  \eqref{dois}, \eqref{sete} and \eqref{dez}, we obtain
$$C_{22}=\frac{D-\tilde{B}\alpha^2_2}{\alpha^2_1(\alpha^2_1-\alpha^2_2)},
\quad C_{24}=0,\qquad
C_{44}=\frac{-D+\tilde{B}\alpha^2_1}{\alpha^2_2(\alpha^2_1-\alpha^2_2)}.$$
We observe that as $$\frac{4(1+\tau^2)}{(4+5\tau^2)} \leq \cos^2\vartheta,$$
then
$$C_{11}=C_{22}<0, \qquad
C_{33}=C_{44}>0.$$

Since $\{C_i\}_{i=1}^4$ are mutually orthogonal  and
$$
||C_1||=||C_2||=\sqrt{-C_{11}},
\qquad
||C_3||=||C_4||=\sqrt{C_{33}},
$$
we obtain a pseudo-orthonormal
 basis of $\r_2^4$
 putting $e_i=C_i/||C_i||$, $i\in\{1,2,3,4\}$,  and we can write:
\begin{eqnarray}\label{30.2}
\gamma(s)=\sqrt{-C_{11}}\, \big(\cos(\alpha_1\,s) e_1+\sin(\alpha_1\,s) e_2\big)+
\sqrt{C_{33}}\,\big( \cos(\alpha_2\,s) e_3+\sin(\alpha_2\,s) e_4\big).
\end{eqnarray}

Now, evaluating in  $s=0$ the identities \eqref{relacoesJ}, we have:
\begin{equation}\label{eq1bis}\begin{aligned}
    &\alpha_2\,C_{33}\langle J_1e_3,e_4\rangle-
    \alpha_1 C_{11}\langle J_1e_1,e_2\rangle\\&
    +\sqrt{-C_{11}C_{33}}\,
    (\alpha_1\langle J_1e_3,e_2\rangle+\alpha_2\langle J_1e_1,e_4\rangle)
    =-\tau^{-1}\cos\vartheta,
    \end{aligned}
\end{equation}
$$
     \langle J_1e_1,e_3\rangle=0\,,
$$
\begin{equation}\label{39}\begin{aligned}
&\alpha_2^3\,C_{33}\langle J_1e_3,e_4\rangle-\alpha_1^3\,
C_{11}\langle J_1e_1,e_2\rangle\\&+
\sqrt{-C_{11}C_{33}}\,(\alpha_1\alpha_2^2\langle J_1e_3,e_2\rangle+
\alpha_1^2\alpha_2\langle J_1e_1,e_4\rangle)=-I,
\end{aligned}
\end{equation}
$$
    \langle J_1e_2,e_4\rangle=0\,,
$$
\begin{equation}\label{eq2bis}
    \alpha_1\langle J_1e_2,e_3\rangle+\alpha_2\langle J_1e_1,e_4\rangle=0\,,
\end{equation}
\begin{equation}\label{eq3bis}
    \alpha_2\langle J_1 e_2,e_3\rangle+\alpha_1\langle J_1e_1,e_4\rangle=0\,.
\end{equation}
We point out that  to obtain the previous identities we have divided by
$\alpha_1^2-\alpha_2^2=\sqrt{b^2(b^2-4a)}$ which is different from zero.
From \eqref{eq2bis} and \eqref{eq3bis}, taking into account the
$\alpha_1^2-\alpha_2^2\neq 0$, it results that
\begin{equation}\label{eq4bis}
     \langle J_1e_3,e_2\rangle=0\,,\qquad \langle J_1e_1,e_4\rangle=0\,.
\end{equation}
Then, $J_1e_1=\pm e_2$ and $J_1e_3=\pm e_4$. So, the position vector
of $\gamma$ is given by
\begin{equation}\label{33.2}
\gamma(s)=\sqrt{-C_{11}}\, \big(\cos(\alpha_1\,s) e_1\pm
\sin(\alpha_1\,s) J_1 e_1\big)+
\sqrt{C_{33}}\,\big( \cos(\alpha_2\,s) e_3\pm \sin(\alpha_2\,s)J_1 e_3\big).
\end{equation}

If we use \eqref{coorT'} for $s=0$, we get $J_1e_1=-e_2$ and
$J_1e_3=-e_4$.

Then, if we fix the orthonormal basis  of $\r^4_2$ given by
$$
\bar{E}_1=(0,0,1,0)\,,\quad \bar{E}_2=(0,0,0,1)\,,\quad
\bar{E}_3=(1,0,0,0)\,,
\quad \bar{E}_4=(0,1,0,0)\,,
$$
there must exists a matrix $A\in \mathrm{O}_2(4)$, with $J_1\,A=-A\,J_1$,
such that $e_i=A\,\bar{E}_i$.
\end{proof}

\section{The case $b^2<4a$}

%
%
%

\begin{theorem}
 Let $\gamma:I \to \S\subset \r^4_2$ be a proper biharmonic curve parametrized by
 arc length,
 such that $b^2<4a$. Then
 \begin{equation}\label{bbb}
 b= -\tau^{-1}(2+\tau^2)\cos\vartheta+\sqrt{(4+5\tau^2)
  \cos^2\vartheta
  -4(1+\tau^2)},
  \end{equation}
  \begin{equation}\label{tt}
  \frac{(2+\tau^2)^2}{4+5\tau^2+\tau^4}<\cos^2\vartheta<1,
  \end{equation}
and the expression of $\gamma$ as a curve in $\r^4_2$ is
\begin{equation}\label{gammaCaso3}
\begin{aligned}
\gamma(s)=&A\Big(\cos\big(\frac{b}{2}\, s\big)\cosh (\mu\, s)
+w_{14}\sin\big(\frac{b}{2}\, s\big)\sinh (\mu\,s\big)\, , \, \\
&\sin\big(\frac{b}{2}\, s\big)\cosh (\mu\, s)
-w_{14}\cos\big(\frac{b}{2}\, s\big)\sinh (\mu\,s\big)\, , \, \\
&\cos\big(\frac{b}{2}\, s\big)\sinh (\mu\, s)\sqrt{1+w_{14}^2}\, , \,
\sin\big(\frac{b}{2}\, s\big)\sinh (\mu\, s)\sqrt{1+w_{14}^2}\Big),
\end{aligned}
\end{equation}
where
$$\mu=\frac{\sqrt{4a-b^2}}{2}, \qquad w_{14}=
\frac{b \tau+2\cos\vartheta}{2\tau \mu}$$
are real constants and $A\in \mathrm{O}_2(4)$ is a $4\times 4$
indefinite orthogonal matrix commuting with $J_1$.
\end{theorem}

\begin{proof}
From $b^2<4a$, it results that $b$ ig given by \eqref{bbb} and $\theta$ satisfies \eqref{tt}.
Also a direct integration of \eqref{11.1}, gives
\begin{equation*}
\gamma(s)=\cos\big(\frac{b}{2}\, s\big)\,\big(\cosh (\mu\, s)\,w_1
+\sinh (\mu\,s\big)\,w_3\big)+
\sin\big(\frac{b}{2}\, s\big)\,\big(\cosh (\mu\, s)\,w_2
+\sinh (\mu\, s)\,w_4\big),
\end{equation*}
where
$$\mu=\frac{\sqrt{4a-b^2}}{2},$$ while the $w_i$,
$i\in\{1,2,3,4\}$,  are constant vectors in $\r^4_2$.
If $w_{ij}:=\langle w_i,w_j\rangle$, evaluating the relations \eqref{relacoes}
in $s=0$, we obtain
\begin{equation}\label{uno2}
    w_{11}=1,
\end{equation}
\begin{equation}\label{due2}
    \frac{b^2}{4}\,w_{22}+\mu^2\,w_{33}+\mu\,b\,w_{23}=\tilde{B},
\end{equation}
\begin{equation}\label{tre2}
    \frac{b}{2}\,w_{12}+\mu \,w_{13}=0,
\end{equation}
\begin{equation}\label{quatro2}
     \frac{b}{2}\,\Big(\mu^2-\frac{b^2}{4}\Big)\,w_{12}+\mu^2 \,b\,w_{34}
     +\mu\,\frac{b^2}{2}\,w_{24}
     +\mu\,\Big(\mu^2-\frac{b^2}{4}\Big)\,w_{13}=0,
\end{equation}
\begin{equation}\label{cinque2}
   \Big(\mu^2-\frac{b^2}{4}\Big)^2\,w_{11}+\mu^2 \,b^2\,w_{44}
     +2\mu\,b\,\Big(\mu^2-\frac{b^2}{4}\Big)\,w_{14}=D,
\end{equation}
\begin{equation}\label{sei2}
   \Big(\mu^2-\frac{b^2}{4}\Big)\,w_{11}+\mu \,b\,w_{14}=-\tilde{B},
\end{equation}
\begin{equation}\label{sette2}
    \frac{b^2}{4}\, \Big(3\mu^2-\frac{b^2}{4}\Big)\,w_{22}+\mu^2 \,
    \Big(\mu^2-3\frac{b^2}{4}\Big)\,w_{33}
     +\mu\,\frac{b}{2}\,(4\mu^2-b^2)\,w_{23}=-D,
\end{equation}
\begin{equation}\label{otto2}\begin{aligned}
   &\frac{b}{2}\, \Big(3\mu^2-\frac{b^2}{4}\Big)\,
   \Big(\mu^2-\frac{b^2}{4}\Big)\,w_{12}+b\,\mu^2 \,
   \Big(\mu^2-3\frac{b^2}{4}\Big)\,w_{34}\\&
   +\mu\,\Big(\mu^2-3\frac{b^2}{4}\Big)\,\Big(\mu^2-\frac{b^2}{4}\Big)w_{13}
   +\mu\, \frac{b^2}{2}\, \Big(3\mu^2-\frac{b^2}{4}\Big)\,w_{24}=0,
     \end{aligned}
\end{equation}
\begin{equation}\label{nove2}
   \frac{b}{2}\,\Big(3\mu^2-\frac{b^2}{4}\Big)\,w_{12}+\mu\,
   \Big(\mu^2-3\frac{b^2}{4}\Big)\,w_{13}
  =0,
\end{equation}
\begin{equation}\label{dieci2}\begin{aligned}
   &\frac{b^2}{4}\, \Big(3\mu^2-\frac{b^2}{4}\Big)^2\,w_{22}+
   \mu^2 \,\Big(\mu^2-3\frac{b^2}{4}\Big)^2\,w_{33}\\&
     +\mu\,b\,\Big(3\mu^2-\frac{b^2}{4}\Big)\,
     \Big(\mu^2-3\frac{b^2}{4}\Big)\,w_{23}=E.
     \end{aligned}
\end{equation}
From  \eqref{uno2}, \eqref{cinque2} and \eqref{sei2}, it follows that
$$w_{11}=-w_{44}=1, \qquad w_{14}=\frac{b\tau+2\cos\vartheta}{2\tau\mu}.$$
Also, from \eqref{tre2} and \eqref{nove2}, we obtain
$$w_{12}=w_{13}=0$$ and, therefore, from \eqref{quatro2} and \eqref{otto2},
$$w_{24}=w_{34}=0.$$
Moreover, using \eqref{due2}, \eqref{sette2} and \eqref{dieci2}, we get
$$w_{22}=-w_{33}=1, \qquad w_{23}=-\frac{b\tau+2\cos\vartheta}{2\tau\mu}.$$
Then we can define the following pseudo-orthonormal basis in
$\r^4_2$:
$$\left\{\begin{aligned}
e_1&=w_1,\\
e_2&=w_2,\\
e_3&=\frac{w_3+w_{14}\,w_2}{\sqrt{1+w_{14}^2}},\\
e_4&=\frac{w_4-w_{14}\,w_1}{\sqrt{1+w_{14}^2}},
\end{aligned}
\right.
$$
with $\langle e_1,e_1\rangle=1=\langle e_2,e_2\rangle$ and
$\langle e_3,e_3\rangle=-1=\langle e_4,e_4\rangle$.

Evaluating the identities \eqref{relacoesJ} in  $s=0$, and taking into account
that:
\begin{equation}
\begin{aligned}\nonumber
&\gamma(0)=w_1\,,\\
&\gamma'(0)=\frac{b}{2}\,w_2+\mu\,w_3\,,\\
&\gamma''(0)=\Big(\mu^2-\frac{b^2}{4}\Big)\,w_1+\mu\, b\,w_4\,,\\
&\gamma'''(0)=\frac{b}{2}\,\Big(3\mu^2 -\frac{b^2}{4})\,w_2
+\mu\,\Big(\mu^2-\frac{3}{4} b^2\Big)\,w_3\,,\\
&\gamma^{IV}(0)=\Big(\mu^4-\frac{3}{2}\mu^2\,b^2
+\frac{b^4}{16}\Big)\,w_1+
2\mu\,b\,\Big(\mu^2-\frac{b^2}{4}\Big)\,w_4\,,
\end{aligned}
\end{equation}
we conclude that
\begin{equation}
\begin{aligned}\nonumber
\langle J_1w_1,w_2\rangle&=-\langle J_1w_3,w_4\rangle=1,\\
\langle J_1w_3,w_2\rangle&=\langle J_1w_1,w_4\rangle=0,\\
\langle J_1w_1,w_3\rangle&=\langle J_1w_2,w_4\rangle=-w_{14}.
\end{aligned}
\end{equation}
Then,
$$\langle J_1e_1,e_2\rangle=-\langle J_1e_3,e_4\rangle=1,$$
$$\langle J_1e_1,e_4\rangle=\langle J_1e_1,e_3\rangle=
\langle J_1e_2,e_3\rangle=\langle J_1e_2,e_4\rangle=0.$$
Therefore, we obtain that
$$J_1e_1=e_2,\qquad J_1e_3=e_4.$$
Consequently, if we consider the orthonormal basis
$\{E_i\}_{i=1}^4$ of $\r^4_2$ given by
$$
E_1=(1,0,0,0)\,,\quad E_2=(0,1,0,0)\,,\quad
E_3=(0,0,1,0)\,,\quad E_4=(0,0,0,1)\,,
$$
there must exists $A\in \mathrm{O}_2(4)$, with $J_1A=AJ_1$, such that
$e_i=A\,E_i$, $i\in\{1,2,3,4\}$.
\end{proof}

\end{document}